\documentclass[12pt,reqno]{article}
\usepackage{graphicx}
\usepackage{amssymb}
\usepackage{amsmath}
\usepackage{amsthm}
\usepackage{amsfonts}
\usepackage{url}
\usepackage{hyperref}
\usepackage{breakurl}
\usepackage[T1]{fontenc}        

\newcommand{\comment}[1]{}
\newcommand{\raisecomma}{\raisebox{2pt}{$,$}}
\newcommand{\raisedot}{\raisebox{2pt}{$.$}}
\newcommand{\sign}{\text{sgn}}
\newcommand{\Dbar}{{\mathcal R}} 
\newcommand{\Had}{{\mathcal H}}

\newcommand{\C}{{\mathbb C}}
\newcommand{\R}{{\mathbb R}}
\newcommand{\N}{{\mathbb N}}

\newcommand{\Z}{{\mathbb Z}}

\newcommand{\Tabexceptions}{{Table~1}}	
\newtheorem{theorem}{Theorem}
\newtheorem{corollary}{Corollary}
\newtheorem{lemma}{Lemma}
\newtheorem{proposition}{Proposition}
\newtheorem{definition}{Definition}
\newtheorem{remark}{Remark}

\begin{document}
\bibliographystyle{plain}
\title{~\\[-40pt]
Lower bounds on maximal determinants\\
 of $\pm 1$ matrices\\ 
 via the probabilistic method
}
\author{Richard P.\ Brent\\
Australian National University\\
Canberra, ACT 0200,
Australia\\
\and
Judy-anne H.\ Osborn\\
The University of Newcastle\\
Callaghan, NSW 2308,
Australia\\
\and
Warren D.\ Smith\\
Center for Range Voting\\			
21 Shore Oaks Drive, Stony Brook\\ NY 11790, USA\\
}

\date{\today}

\maketitle
\thispagestyle{empty}                   

\begin{abstract}
We show that the maximal determinant
$D(n)$ for $n \times n$ $\{\pm 1\}$-matrices 
satisfies $\Dbar(n) := D(n)/n^{n/2} \ge \kappa_d > 0$.
Here $n^{n/2}$ is the Hadamard upper bound,
and $\kappa_d$ depends only on $d := n-h$, 
where $h$ is the maximal order of
a Hadamard matrix with $h \le n$. 
Previous lower bounds on $\Dbar(n)$ depend on both $d$ and
$n$. Our bounds are improvements, 
for all sufficiently large $n$, if $d > 1$.

We give various lower bounds
on $\Dbar(n)$ that depend only on~$d$.
For example,
$\Dbar(n) \ge 0.07\,(0.352)^d > 3^{-(d+3)}$.
For any fixed $d \ge 0$ we have
$\Dbar(n) \ge (2/(\pi e))^{d/2}$
for all sufficiently large $n$ (and conjecturally for all positive~$n$).
If the Hadamard conjecture is true, then
$d \le 3$ and $\kappa_d \ge (2/(\pi e))^{d/2} > 1/9$.
\end{abstract}

\pagebreak[4]
\section{Introduction}		\label{sec:intro}

Let $D(n)$ be the maximal determinant possible for an $n\times n$
matrix with elements drawn from the real interval $[-1, 1]$.
Hadamard~\cite{Hadamard}
\footnote{For
earlier contributions by Desplanques, 
L\'evy, 
Muir, 
Sylvester 
and Thomson (Lord Kelvin), 
see~\cite{Muir26,Sylvester} and~\cite[pg.~384]{MS}.}
proved that $D(n) \le n^{n/2}$, and the
\emph{Hadamard conjecture} is that a matrix achieving this upper bound
exists for each positive integer $n$ divisible by four.
The function $\Dbar(n) := D(n)/n^{n/2}$ 
is a measure of the sharpness of the Hadamard bound.
Clearly $\Dbar(n) = 1$ if a Hadamard matrix of order $n$
exists; otherwise $\Dbar(n) < 1$.  The aim of this paper is
to give lower bounds on $\Dbar(n)$.

If $h \le n$ is the order of a Hadamard matrix, and $d=n-h$, then
we show that $\Dbar(n)$ is bounded below by a positive constant
$\kappa_d$ (depending on $d$ but not on $n$).
When $d>1$ 
this improves on previous results\footnote{See \cite[Theorem~9]{rpb249} 
and the
references cited there.  For example, the well-known bound of
Clements and Lindstr\"om~\cite[Corollary to Thm.~2]{CL65} only shows that
$\Dbar(n) > (3/4)^{n/2}$.}
for which the lower
bound was (at best) of order $n^{-\alpha d}$ for some constant
$\alpha \ge 1/2$.
Rokicki \emph{et al}~\cite{Orrick-prog} conjectured that
$\Dbar(n) \ge 1/2$ on the basis of computational results
for $n \le 120$.  

We obtain lower bounds on $\Dbar(n)$ using the probabilistic
method pioneered by Erd\H{o}s (see for example~\cite{AS,ES}).
Specifically, we adjoin $d$ extra columns to the $h \times h$ Hadamard matrix,
and fill their $h \times d$ entries with random signs
obtained by independently tossing fair coins.
Then we adjoin $d$ extra rows, and fill their $d \times (h+d)$ entries with
$\pm 1$ signs chosen deterministically 
in a way intended to approximately 
maximize the determinant of the final matrix.
To do so, we use the fact 
that this 
determinant can be expressed in terms 
of the  $d \times d$ Schur complement
(see~\S\ref{sec:Schur}).
In the proof of Theorem~\ref{thm:unconditional_bound} we obtain a lower
bound on the expected value of the determinant in a direct manner.
In the proofs of Theorems~\ref{thm:unconditional_bound2}
and~\ref{thm:unconditional_bound3}
we use 
a Hoeffding tail bound to show that the
Schur complement is, with high probability,
sufficiently diagonally dominant that its determinant is close
to the product of its diagonal elements. 
We employ two possibly new inequalities,
Lemma~\ref{lemma:optimal_pert_bound} and
Lemma~\ref{lemma:nice_pert_bound} in \S\ref{sec:lemmas},
that give lower bounds on the 
determinant of a diagonally dominant matrix.
The bounds are sharper than the obvious bounds arising from
Gerschgorin's circle theorem~\cite{Gerschgorin,Varga},
so may be of independent interest.

\comment{
So far, the result has only been probabilistic, but since the success
probability is positive, that tells us that a suitable configuration of coin
tosses must {\it deterministically exist} so that an $n \times n$ sign
matrix indeed must exist satisfying our lower bound.  (And indeed the
argument could be derandomized using the ``method of conditional
expectations'' which would show there is a polynomial-time algorithm to find
a suitable $n \times n$ matrix?)  WDS

NB: I believe this if $d=1$ but I'm not sure about it if $d > 1$ because
of our way of handling the off-diagonal elements in the Schur complement.
That's why I've commented out this paragraph. RPB
} 

\comment{
We found several different proofs all with this same basic
strategy,
with different amounts of complexity and different strengths in various parts of 
their arguments, and the ones we inserted into this paper were selected somewhat
arbitrarily from our collection.
} 

In the special case $d=1$ our argument simplifies,
because there is no need to consider a nontrivial Schur complement or to
deal with the contribution of the off-diagonal elements.
This case was (essentially) already considered by
Brown and Spencer~\cite{BS}, Erd\H{o}s and Spencer~\cite[Ch.~15]{ES}, 
and (independently) by 
Best~\cite{Best}; see also~\cite[\S2.5]{AS}
and ~\cite[Problem A4]{Putnam74}.
The consequence for lower bounds on $\Dbar(n)$ when
$n \equiv 1 \bmod 4$ was exploited by
Farmakis and Kounias~\cite{FK}, 
and an improvement using $3$-normalized Hadamard
matrices was considered by Orrick and Solomon~\cite{OS07}.

In \S\ref{sec:gaps} we review previous results that give upper bounds on
gaps between the orders of Hadamard matrices.
These are relevant as they enable us to bound $d = n-h$ 
as a function of~$h$.  

Various preliminary results are proved in \S\ref{sec:lemmas},
and the main results are proved in \S\ref{sec:unconditional}.
Theorem~\ref{thm:unconditional_bound}
applies for fixed $d$ and $h \ge h_0(d)$,
where the function $h_0(d)$ 
grows rapidly, but this is not significant for the cases $d\le 3$
that arise if we assume the Hadamard conjecture.
For $d \le 3$, Corollary~\ref{cor:small_d}
shows that $\Dbar(n)$ is bounded below by 
$(2/(\pi e))^{d/2} > 1/9$, coming close to
Rokicki \emph{et al}'s conjectured lower bound of $1/2$,
and improving on earlier results~\cite{rpb249,CL65,Cohn63,KMS00,LL}
that failed to obtain a constant lower bound on $\Dbar(n)$ for 
$2 \le d \le 3$.

At the cost of more complicated proofs,
Theorems~\ref{thm:unconditional_bound2} and~\ref{thm:unconditional_bound3} 
apply to larger regions of $(d,h)$-space.
Theorem~\ref{thm:unconditional_bound2}
applies for $h/\ln h \ge 16d^3$,
and Theorem~\ref{thm:unconditional_bound3}
applies for $h \ge 6d^3$.
In view of known results on gaps between
Hadamard orders, discussed in \S\ref{sec:gaps}, these theorems
give a lower bound on $\Dbar(n)$ for all but a finite set $E$ of 
positive integers~$n$.
We have obtained a lower bound on $\Dbar(n)$ for each $n\in E$ by
explicit computation, using a probabilistic algorithm that uses the
same construction as the proofs of these theorems.
This leads to Theorem~\ref{thm:no_exceptions}, which gives a lower bound
$\Dbar(n) > 3^{-(d+3)}$ that is valid for all positive integers~$n$
(the constants here are not the best possible).

\subsubsection*{Acknowledgements}
We thank Robert Craigen for informing us of the work of his student Ivan
Livinskyi, and Will Orrick for his comments and for providing a copy of the
unpublished report~\cite{Orrick-prog}. Dragomir {\DJ}okovi\'c and Ilias
Kotsireas shared their list of known small Hadamard orders, which was very
useful for checking the program that we used in the proof of
Lemma~\ref{lemma:hadregion}.

\section{Gaps between Hadamard orders}	\label{sec:gaps}

In order to apply our results to obtain a lower bound on $\Dbar(n)$ for
given $n$, we need to know the order $h$ of a Hadamard matrix with
$h \le n$ and $n-h$ preferably as small as possible.  Thus, it is of
interest to consider the size of possible gaps in the sequence 
$(n_i)_{i \ge 1}$ of Hadamard
orders.  
We define the \emph{Hadamard gap function} $\gamma:\R \to \Z$ by
\begin{equation}	\label{eq:gamma_defn}
\gamma(x) := \max \{n_{i+1}-n_i \,|\, n_{i} \le x\} \cup \{0\}\,.
\end{equation}
In~\cite{rpb249} it was shown,
using the Paley and Sylvester constructions, 
that $\gamma(n)$ can be bounded using the prime-gap function.
For example, if $p$ is an odd prime, then $2(p+1)$ is a Hadamard order.
However, only rather weak bounds on the prime-gap function are known.
A different approach which produces asymptotically-stronger 
bounds employs results of 
Seberry~\cite{Seberry}, as subsequently sharpened
by Craigen~\cite{Craigen1}, Livinskyi~\cite{Livinskyi}, and 
Smith~\cite{Smith-epsilon}. 
These results take the following form: 
for any odd positive
integer~$q$, a Hadamard matrix of order $2^t q$ exists for every integer
\[t \ge \alpha \log_2(q) + \beta,\]
where $\alpha$ and $\beta$ are author-dependent constants.
Seberry~\cite{Seberry} obtained $\alpha=2$. 
Craigen~\cite{Craigen1}
improved this to $\alpha = 2/3$, $\beta = 16/3$,
and later obtained 
$\alpha = 3/8$ 
in unpublished 
work with Tiessen quoted
in~\cite[Thm.~2.27]{Horadam}
and \cite{Craigen2,CK}.%
\footnote{There are typographical errors
in~\cite[Thm.~2.27]{Horadam} and
in~\cite[Thm.~1.43]{CK}, 
where the floor function should
be replaced by the ceiling function. This has the effect of increasing
the additive constant $\beta$.}
Livinskyi~\cite{Livinskyi} found $\alpha = 1/5$, $\beta = 64/5$.
Smith's unpublished paper~\cite{Smith-epsilon} 
shows that $\gamma(n) = O(n^\varepsilon)$
for each $\varepsilon > 0$, but the constants hidden in the ``$O$'' 
in this result can be very large, so we do not use Smith's result here.

The connection between these results and the
Hadamard gap function is given by Lemma~\ref{lemma:conversion}.  
{From} the lemma
and the results of Livinskyi, the Hadamard gap function satisfies
\begin{equation}	\label{eq:gamma_Livinskyi}
\gamma(n) = O(n^{1/6}).
\end{equation}
This is much sharper than  $\gamma(n) = O(n^{21/40})$ 
arising from the best current result for prime gaps (by
Baker, Harman and Pintz~\cite{BHP}),
although not as sharp as the result
$\gamma(n) = O(\log^2 n)$ that would follow from
Cram\'er's prime-gap conjecture~\cite{rpb249,Cramer,Shanks64,Silva}.
\begin{lemma}	\label{lemma:conversion}
Suppose there exist constants $\alpha$, $\beta$ such that, 
for any odd positive integer $q$, a Hadamard matrix of order $2^t q$
exists for all $t \ge \alpha \log_2(q) + \beta$.
Then the Hadamard gap function $\gamma(n)$ satisfies
\[\gamma(n) = O(n^{\alpha/(1+\alpha)})\,.
\]
\end{lemma}
\begin{proof}
Consider consecutive odd integers $q_0$, $q_1 = q_0+2$ and corresponding
$n_i = 2^tq_i$, where $t = \lceil\alpha \log_2(q_1) + \beta\rceil$. 
By assumption
there exist Hadamard matrices of orders $n_0$, $n_1$.
Also, $2^\beta q_1^\alpha \le 2^t < 2^{\beta+1}q_1^\alpha$.
Thus \[n_1 = 2^tq_1 \ge 2^\beta q_1^{1+\alpha}\]
and $\;\;\;\;n_1-n_0 
 = 2^{t+1} < 2^{\beta+2}q_1^\alpha
 \le 2^{2+\beta/(1+\alpha)}n_1^{\alpha/(1+\alpha)}
  = O(n_0^{\alpha/(1+\alpha)})$.
\end{proof}

\section{The Schur complement}	\label{sec:Schur}

Let \[\widetilde{A} = \left[\begin{matrix} A & B\\ C & D\\
\end{matrix}\right] \] be an $n \times n$ matrix written in 
block form, where $A$ is $h \times h$, and $n = h + d > h$.  Then
the \emph{Schur complement}~\cite{Schur}
of $A$ in $\widetilde{A}$ is the $d \times d$
matrix 
\[D - CA^{-1}B.\] 
The Schur complement is relevant to our problem due to the following lemma.
\begin{lemma}	\label{lemma:Schur}
If $\widetilde{A}$ is as above, with $A$ nonsingular, then
\[\det(\widetilde{A}) = \det(A)\det(D - CA^{-1}B).\]
\end{lemma}
\begin{proof}
Using block Gaussian elimination on $\widetilde{A}$ gives
\[
\left[\begin{matrix} A & B\\ C & D\\ \end{matrix}\right] =
\left[\begin{matrix} I & 0\\ CA^{-1} & I\\ \end{matrix}\right]
\left[\begin{matrix} A & B\\ 0 & D-CA^{-1}B\\ \end{matrix}\right]\,.
\]
Now take determinants.
\end{proof}

\section{Notation and auxiliary results}	\label{sec:lemmas}

In this section we define our notation and prove 
some auxiliary results that are needed
in~\S\ref{sec:unconditional}.
As above, $D(n)$ is the maximum determinant function
and $\Dbar(n) := D(n)/n^{n/2}$ is its normalization by
the Hadamard bound $n^{n/2}$.
The set of orders of all Hadamard matrices is denoted by ${\Had}$.

We define $c := \sqrt{2/\pi} \approx 0.7979$.  Other constants are denoted
$c_1$, $c_2$, $\alpha$, $\beta$, etc.
Usually $h\in\Had$ and $n = h+d$,
where $d\ge 0$ (the case $d=0$ is trivial because then the Hadamard bound
applies). We assume
$h\ge 4$ to avoid the cases $h \in \{1,2\}$, although in most cases
it is easy to verify that the results also hold for $h \in \{1,2\}$.

Matrices are denoted by capital letters $A$ etc, and their elements by the
corresponding lower-case letters, e.g. $a_{ij}$ (the comma between
subscripts is omitted if the meaning is clear).
 
When using the probabilistic method, the probability of an event $S$
(which is always a discrete set of possible outcomes of a random process) 
is denoted by ${\rm Pr}(S)$,
and the expectation of a random variable $X$ is denoted by $E(X)$.

\begin{lemma}	\label{lemma:2n_choose_n}
Suppose that $h$ is an even positive integer.Then
\[
\binom{h}{h/2} > 
 \,2^h\, \sqrt{\frac{2}{\pi h}}\left(1 - \frac{1}{4h}\right).
\]
\end{lemma}
\begin{proof}
This follows from Stirling's asymptotic expansion of $\ln\Gamma(x)$ with
the error bounded by the first term omitted,
see for example~\cite[eqn.~(4.38)]{rpb226}.
\end{proof}
\begin{lemma}	\label{lemma:g_bound2}
Let $g(h) := 1 + 2^{-h}{h}\binom{h}{h/2}$,
where $h \ge 4$ is an even integer.
Then $g(h) > ch^{1/2} + 1 - ch^{-1/2}/4$ and
$g(h) > ch^{1/2} + 0.9$, where $c = \sqrt{2/\pi}$.
\end{lemma}
\begin{proof}
The first inequality follows from Lemma~\ref{lemma:2n_choose_n}.
{From} the condition $h \ge 4$, we have $ch^{-1/2}/4 < 1/10$.
Thus $g(h) > ch^{1/2} + 1 - 0.1 = ch^{1/2} + 0.9$.
\end{proof}
\pagebreak[3]

\noindent Lemma~\ref{lemma:ineq1} is from~\cite[Lemma~4]{rpb249},
and Lemma~\ref{lemma:uncond2} is similar.
\begin{lemma} \label{lemma:ineq1}
If $\alpha \in \R$, $n \in \N$, $n > |\alpha| > 0$, and $h = n - \alpha$, then
\[\frac{h^h}{n^n} >
\left(\frac{1}{ne}\right)^\alpha\,\raisedot\]
\end{lemma}
\begin{proof}
Taking logarithms, and writing $x = \alpha/n$,
the inequality reduces to
\begin{equation}	\label{eq:x_ineq}
(1-x)\ln(1-x) + x > 0,
\end{equation}
or equivalently (since $0 < |x| < 1$)
\[\frac{x^2}{1\cdot 2} + \frac{x^3}{2\cdot 3} + \frac{x^4}{3\cdot 4} +
\cdots > 0.\]
This is clear if $x>0$, and also if $x < 0$ because then the terms alternate
in sign and decrease in magnitude.
\end{proof}

\begin{lemma} \label{lemma:uncond2}
If $\alpha \in \R$, $n \in \N$, $n > |\alpha| > 0$, and $h = n - \alpha$, then
\[(h/n)^n > \exp(-\alpha - \alpha^2/h).\]
\end{lemma}
\begin{proof}
Taking $x = \alpha/n$, the inequality~\eqref{eq:x_ineq} proved above implies
that\\ $\ln(1-x) > -x/(1-x)$, so
\[(1-x)^n \,>\, \exp\left(-\frac{nx}{1-x}\right)\,.\]
Since $1-x = h/n$, we obtain
\[\left(\frac{h}{n}\right)^n  
 \,>\, \exp\left(- \frac{\alpha}{1-\alpha/n}\right) 
  = \exp(-\alpha - \alpha^2/h).\]
\end{proof}

\begin{lemma}	\label{lemma:u_sum}
Let $A \in \{\pm1\}^{h\times h}$ be a Hadamard matrix,
$C \in \{\pm1\}^{d\times h}$, and $U = CA^{-1}$.  
Then, for each $i$ with $1 \le i \le d$,
\[\sum_{j=1}^h u_{ij}^2 = 1.\]
\end{lemma}
\begin{proof}
Since $A$ is Hadamard, $A^TA = hI$. Thus
$UU^T = 
h^{-1}CC^T$.
Since $c_{ij} = \pm 1$, ${\rm diag}(CC^T) = hI$.
Thus ${\rm diag}(UU^T) = I$.
\end{proof}
\begin{definition}	\label{defn:DD}
If $A \in \R^{d\times d}$ satisfies
$|a_{ij}| \le \varepsilon |a_{ii}|$
for all $i \ne j$, then
we say that $A$ is \emph{${\rm DD}(\varepsilon)$}.
{\rm (Here ``DD'' stands for ``diagonally dominant''.)}
\end{definition}
\begin{lemma}	\label{lemma:optimal_pert_bound}
If $A = I - E \in \R^{d\times d}$, 
$|e_{ij}| \le \varepsilon$ for $1 \le i, j \le d$,
and $d\varepsilon \le 1$, then 
\[\det(A) \ge 1 - d\varepsilon.\]
\end{lemma}
\begin{proof}
We first assume that $d\varepsilon < 1$.
Thus, by Gerschgorin's theorem,
$A$ is nonsingular. 
Hence by continuity $\det(A) > 0$. 
Thus, $\ln\det(A)$ is well-defined and real.  
Write the eigenvalues of $X \in \R^{d\times d}$ as $\lambda_i(X) \in \C$,
and define the \emph{trace}
${\rm Tr}(X) := \sum_i x_{ii} = \sum_i \lambda_i(X)$.
Then
\[\ln\det(A) = \ln \left(\prod_{i=1}^d \lambda_i(A)\right)
   = {\rm Tr}(\ln(A)),
\]
where 
\[\ln(A) = \ln(I-E) = -\sum_{k=1}^\infty \frac{1}{k}E^k\,.\]
Thus
\[\ln\det(A) 
 = -{\rm Tr}\left(\sum_{k=1}^\infty\frac{1}{k}E^k\right)
 = - \sum_{k=1}^\infty \frac{1}{k}{\rm Tr}(E^k)\,.\]
Considering this series term by term,
it is clear that ${\rm Tr}(E^k)$ attains its maximum value, subject to
the constraints $|e_{ij}| \le \varepsilon$, when each $e_{ij} = \varepsilon$,
that is when $E = E_1 := \varepsilon\, ee^T$,
where $e^T := (1, 1,\ldots, 1)$ is the $d$-vector of all ones.
Using $e^Te = d$, 
it is easy to prove, by induction on $k$, that
$E_1^k = (d\varepsilon)^{k-1}E_1$ for all $k \ge 1$.
Thus ${\rm Tr}(E_1^k) = (d\varepsilon)^{k-1}{\rm Tr}(E_1) =
 (d\varepsilon)^{k}$.
So we have
\[\ln\det(A) \ge -\sum_{k=1}^\infty \frac {(d\varepsilon)^k}{k}
 = \ln(1 - d\varepsilon)\,,\]
and it follows that
$\det(A) \ge 1 - d\varepsilon$.
This completes the proof for $d\varepsilon < 1$. 
If $d\varepsilon = 1$
then $\det(A) \ge 0$ by a continuity argument.
\end{proof}
\begin{remark}
{\rm
It is easy to show, using a rank-$1$ updating formula, that
\[\det(I - \varepsilon\,ee^T) = 1 - d\varepsilon\,.\] 
Thus, the inequality of Lemma~\ref{lemma:optimal_pert_bound} 
is best possible. 
One may see from
the proof of Lemma~\ref{lemma:optimal_pert_bound}
that if $\varepsilon > 0$
then tightness occurs only for $E=\varepsilon\,ee^T$.
In this unique extreme case, the eigenvalues
of $A = I - E$ 
are $1-d\varepsilon$ (with multiplicity $1$)
and $1$ (with multiplicity $d-1$).
}
\end{remark}
\begin{remark}
{\rm
Gerschgorin's theorem gives $|\lambda_i(A) - 1| \le d\varepsilon$,
but this only implies the much weaker inequality
$\det(A) \ge (1 - d\varepsilon)^d$.
}
\end{remark}
If, in addition to the conditions of Lemma~\ref{lemma:optimal_pert_bound},
we assume that $e_{ii} = 0$, then in the extreme case the eigenvalues
of $A$ are all shifted up by $\varepsilon$. Thus we obtain the 
following lemma. The proof is omitted since it is similar to the proof of
Lemma~\ref{lemma:optimal_pert_bound}.
\begin{lemma}	\label{lemma:optimal_pert_bound2}
If $A = I - E \in \R^{d\times d}$, 
$|e_{ij}| \le \varepsilon$ for $1 \le i, j \le d$,
$e_{ii} = 0$ for $1 \le i \le d$,
and $(d-1)\varepsilon \le 1$, then 
\[\det(A) \ge \left(1 - (d-1)\varepsilon\right)(1+\varepsilon)^{d-1}.
\]
\end{lemma}

The following lemma, which may be of independent interest,
gives a lower bound on the determinant of a
diagonally dominant matrix.
\begin{lemma}	\label{lemma:nice_pert_bound}
If $A \in \R^{d\times d}$ is ${\rm DD}(\varepsilon)$, then
\[|\det(A)| \ge \left(\prod_{i=1}^d |a_{ii}|\right)
 \left(1-(d-1)^2\varepsilon^2\right).\]
\end{lemma}
\begin{proof}
If $\varepsilon < 0$ then $A=0$ and the result is trivial;
if $(d-1)\varepsilon \ge 1$ then the inequality is trivial as the right
side is not positive. 
Hence, assume that $0 \le (d-1)\varepsilon < 1$.
If any $a_{ii} = 0$ then the result is trivial. Otherwise,
apply Lemma~\ref{lemma:optimal_pert_bound2} to $SA$, 
where $S = {\rm diag}(a_{ii}^{-1})$.
Since $\det(A) = \det(SA)\prod_i a_{ii}$ and
\[(1-(d-1)\varepsilon)(1+\varepsilon)^{d-1} \ge
 (1-(d-1)\varepsilon)(1 + (d-1)\varepsilon) =
 1 - (d-1)^2\varepsilon^2,\]
the corollary follows. 
\end{proof}

\begin{remark}
{\rm
Lemma~\ref{lemma:nice_pert_bound} is much sharper than the bound
\[ |\det(A)| \ge \left(\prod_{i=1}^d |a_{ii}|\right)
   \left(1-(d-1)\varepsilon\right)^d \]
that follows from Gerschgorin's theorem.
For example, if $a_{ii} = 1$ for $1\le i\le d$ and $(d-1)\varepsilon = 1/2$,
then Lemma~\ref{lemma:nice_pert_bound} gives the lower bound $3/4$
whereas Gerschgorin's theorem gives $2^{-d}$.
}
\end{remark}

\begin{lemma}	\label{lemma:exp_ineq}
If $\kappa, \varepsilon_0 \in \R$, 
$\varepsilon_0 > 0$, $|\kappa \varepsilon_0| < 1$, then
$1+\kappa \varepsilon \ge \exp(\beta\varepsilon)$ for all
$\varepsilon\in[0,\varepsilon_0]$, where
\[\beta = \frac{\ln(1+\kappa \varepsilon_0)}{\varepsilon_0}\,\raisedot\]
\end{lemma}
\begin{proof}
This follows from the concave-up 
nature of $\exp(K\varepsilon)$, and the fact that
$1+\kappa \varepsilon = \exp(\beta\varepsilon)$ 
at the two endpoints $\varepsilon=0$ and $\varepsilon=\varepsilon_0$.
\end{proof}

The following lemma is essentially
Erd\H{o}s and Spencer~\cite[Lemma 15.2]{ES}, so
we omit the (straightforward) proof.
\begin{lemma}	\label{lemma:ES15.2}
If $X\in[0,1]$ is a random variable with $E(X) = \mu$,
then for $\lambda < \mu$ we have
\[{\rm Pr}(X \ge \lambda) \ge \frac{\mu - \lambda}{1 - \lambda}\,\raisedot\]
\end{lemma}

We now state a two-sided version of Hoeffding's ``tail inequality.''
For a proof, see~\cite[Theorem 2]{Hoeffding}.
\begin{proposition} 	\label{prop:Hoeffding}
Let $X_1,\ldots,X_h$ be independent random variables
with sum
$Y = X_1 + \cdots + X_h$. 
Assume that
$X_i \in [a_i, b_i]$.
Then, for all $t > 0$,
\[
{\rm Pr}\left(|Y - E[Y]| \ge t\right) \;\le\; 
 2\,\exp\left(\frac{-2t^2}{\sum_{i=1}^h (b_i-a_i)^2} \right)\,\raisedot
\]

\end{proposition}

\pagebreak[3]
\section{Lower bounds on $D(n)$ and $\Dbar(n)$} \label{sec:unconditional}

In this section we prove several lower bounds on $D(n)$ and $\Dbar(n)$,
where $n = h+d$ and $h$ is the order of a Hadamard matrix.
Theorem~\ref{thm:unconditional_bound} applies when $h \ge h_0(d)$ is 
sufficiently large. If we assume the Hadamard conjecture, then 
we can drop the ``sufficiently large'' restriction (see
Corollary~\ref{cor:small_d}).

If the Hadamard conjecture is false then it is sometimes necessary
to take \hbox{$d\ge 4$}.  In this case 
Theorems~\ref{thm:unconditional_bound2}
and~\ref{thm:unconditional_bound3}
are preferable as they impose weaker restrictions on $h$
than does Theorem~\ref{thm:unconditional_bound},
at the cost of a slight weakening of the lower bound on $D(n)$.
The proofs of Theorems~\ref{thm:unconditional_bound2}
and~\ref{thm:unconditional_bound3}
use Lemma~\ref{lemma:nice_pert_bound} and 
Proposition~\ref{prop:Hoeffding},
which are not needed for the proof of Theorem~\ref{thm:unconditional_bound}.

\begin{theorem}	\label{thm:unconditional_bound}
If $d \ge 1$, $h \in {\Had}$, 
$n = h+d$, and 
\begin{equation}	\label{eq:h_0}
h \ge h_0(d) := \left(e(\pi/2)^{d/2} (d-1)! + {d}\right)^2,
\end{equation}
then
\begin{equation}	\label{eq:first_uncond_bd}
\frac{D(n)}{h^{h/2}} > \left(\frac{2n}{\pi}\right)^{d/2}.
\end{equation}
\end{theorem}
\begin{proof}
Let $A$ be a Hadamard matrix of order $h\ge 4$.  
We add a border of $d$ rows
and columns to give a larger matrix $\widetilde{A}$ of order $n$. The 
border is defined by matrices $B$, $C$ and $D$ as in \S\ref{sec:Schur}.
The matrices
$A$, $B$, $C$, and $D$ all have entries drawn from $\{\pm1\}$.
We show that a suitable choice of $B$, $C$ and $D$ gives a matrix
$D - CA^{-1}B$ with sufficiently large determinant that the result can
be deduced from Lemma~\ref{lemma:Schur}.

Define $M = F-D$, where $F = CA^{-1}B$.  Thus $-M$ is the 
Schur complement of $A$ in $\widetilde{A}$.
Note that, since $A$ is a Hadamard matrix, $A^T = hA^{-1}$.

Following Best's approach, $B$ is allowed to range over the set
$S(h,d)$ of all $h\times d$ $\{\pm 1\}$-matrices. We give a lower
bound on the mean value \hbox{$\mu := E(\det(M))$} and deduce that a matrix
$B$ exists for which $\det(M) \ge \mu$.  
We use $E(\cdots)$ to denote a mean value over all possible choices
of $B \in S(h,d)$, unless the mean value over some subset of $S(h,d)$
is specified.
  
The $d \times h$ matrix 
$C = (c_{ij})$ depends on $B$. We choose
\[c_{ij} = \sign (B^T A)_{ij}\,,\]
where  
\[\sign(x) := \begin{cases} +1 \text{ if } x \ge 0,\\
			   -1 \text{ if } x < 0.\\
	     \end{cases}
\]
\noindent
[\emph{Remark}. 
The choice of $C$ ensures that there is no cancellation in the inner
products defining the diagonal entries of $hF = C\cdot (A^T B)$. Thus, we
expect the diagonal entries $f_{ii}$ of $F$ to be nonnegative and 
of order $h^{1/2}$,
but the off-diagonal entries $f_{ij}$ ($i\ne j$) 
to be of order unity with high probability.]

Best~\cite[Theorem 1]{Best} shows\footnote{
In~\cite[footnote on pg.~68]{ES} this result is attributed to J.~H.~Lindsey.
The upper bound can
be achieved infinitely often, in fact
whenever a regular Hadamard matrix of order $h$ exists.
For example, this is true if $h = 4q^2$, 
where $q$ is an odd prime power
and $q \not\equiv 7 \pmod 8$, see~\cite{XXS}.},
\comment{Also,
J.-M. Goethals and J. J. Seidel [Canad. J. Math. 22 (1970), 597–614;
MR0282872 (44 \#106)] showed that ``if there is an Hadamard matrix of
order $n$ then there exists a regular symmetric Hadamard matrix with
constant diagonal of order $n^2$.''
However, the construction does not always give a maxdet matrix,
e.g. taking $h=36$ it gives the ``old record'' of $63×9^{17}×2^{36}$
and not the optimal value of $72×9^{17}×2^{36}$ which can be obtained
by the Orrick-Solomon $3$-normalised construction.
Similarly for $h=100$.
} 
using the Cauchy-Schwarz inequality, 
that $0 \le f_{ii} \le h^{1/2}$, and it follows
similarly that $|f_{ij}| \le h^{1/2}$.

We take $D = (d_{ij})_{1 \le i, j \le d}$ to be a $d \times d$ matrix
with diagonal entries $d_{ii} = -1 $ and off-diagonal entries
to be specified later.

Let $g(h)$ be as in Lemma~\ref{lemma:g_bound2}.
Observe that
\[
E(f_{ij}) = \begin{cases} g(h)-1 \text{ if } i = j,\\
			     0 \text{ otherwise,}\\
             \end{cases}
\]
where the case $i=j$ follows from Best~\cite[Theorem~3]{Best}.
We now show that
\begin{equation} 	\label{eq:Efij2}
E(f_{ij}^2) = 1 \; \text{ if } \; i \ne j.
\end{equation} 
To prove this, assume without essential loss of generality that 
$i=1$, $j > 1$.  
Write $F = UB$, where $U = CA^{-1} = h^{-1}CA^T$.
Now \[f_{1j} = \sum_k u_{1k}b_{kj},\]
where
\[u_{1k} = \frac{1}{h}\sum_{\ell} c_{1\ell}a_{k\ell}\]
and
\[c_{1\ell} = {\rm sgn}\left(\sum_m b_{m1}a_{m\ell}\right).\]
Observe that $c_{1\ell}$ and $u_{1k}$ depend only on the first column of $B$.
Thus, $f_{1j}$ depends only on the first and $j$-th columns of $B$.
If we fix the first column of $B$ and take expectations over all choices
of the other columns, we obtain
\[E(f_{1j}^2) = E\left(\sum_k\sum_{\ell}u_{1k}u_{1\ell}b_{kj}b_{\ell j}\right).\]
The expectation of the terms with $k\ne \ell$ vanishes,
and the expectation of the terms with $k=\ell$
is $\sum_k u_{1k}^2$.
Thus, \eqref{eq:Efij2} follows from Lemma~\ref{lemma:u_sum}.

Now suppose that $i \ne j$, $k \ne \ell$. 
We cannot 
assume that $f_{ij}$
and $f_{k\ell}$ are independent\footnote{For example,
$f_{12}$ and $f_{21}$ are
not independent. Since $f_{ij}$ depends on columns $i$ and $j$ of
$B$, we see that $f_{ij}$ and $f_{k\ell}$ are independent iff
$\{i,j\}\cap\{k,\ell\} = \emptyset$.}. 
However, from the Cauchy-Schwarz inequality, we have
\begin{equation}	\label{eq:E_fijfkl}
E(|f_{ij}f_{k\ell}|) \le \sqrt{E(f_{ij}^2)E(f_{k\ell}^2)} = 1.
\end{equation}
Since $f_{ii}$ depends only on the $i$-th column of $B$,
the ``diagonal'' terms $f_{ii}$ are independent;
similarly the diagonal terms $m_{ii}$ are independent.
Now $E(m_{ii}) = E(f_{ii}) - d_{ii} = g(h)$ by our choice $d_{ii} = -1$, so
\[
E\left(\prod_{i=1}^d m_{ii}\right) = \prod_{i=1}^d E(m_{ii}) = g(h)^d.
\]

Observe that $\det(F+I)$ is the sum of a ``diagonal''
term $\prod_{1 \le i \le d}m_{ii}$ and \hbox{$(d!-1)$}
``non-diagonal'' terms.  
If $d>1$, the non-diagonal
terms each contain at most \hbox{$d-2$} factors of the form $m_{ii}$
(bounded by $h^{1/2}+1$)
and at least two factors of the form $f_{ij}$. 
The expectations of the non-diagonal terms are bounded by 
$(h^{1/2}+1)^{d-2}$.
For example, if $d=3$, we use
\[|E(f_{12}f_{21}m_{33})| \le E(|f_{12}f_{21}|)\max(|m_{33}|) \le
h^{1/2}+1.\]
In general, we use an upper bound $h^{1/2}+1$ for $d-2$ of the factors,
and save a factor of order $h$ by using~\eqref{eq:E_fijfkl} once.

Thus
\begin{equation}	\label{eq:better_ineq}
E(\det(F+I)) \ge g(h)^d - (d!-1)\,(h^{1/2}+1)^{d-2}.
\end{equation}
We simplify~\eqref{eq:better_ineq} using
$h^{1/2}+1 \le h^{1/2}\exp(h^{-1/2})$ and, from~\eqref{eq:h_0}, $d < h^{1/2}$.
Thus
$(h^{1/2}+1)^{d-2} \le h^{d/2-1}\exp(dh^{-1/2})
 \le h^{d/2-1}e$, and~\eqref{eq:better_ineq} gives 
\begin{equation}	\label{eq:better_ineq2}
E(\det(F+I)) \ge g(h)^d - d!\,h^{d/2-1}e.
\end{equation}
Now, using Lemma~\ref{lemma:g_bound2} gives
\begin{equation}	\label{eq:alpha}
E(\det(F+I)) > (ch^{1/2} + 0.9)^d - d!\,h^{d/2-1}e
	\ge c^dh^{d/2}\left(1 + \frac{0.9d}{ch^{1/2}} -
		\frac{d!e}{c^d h}\right)\,.
\end{equation}
We also have
\begin{equation}	\label{eq:beta}
\left(\frac{h}{n}\right)^{d/2} =\; 
  \left(1+\frac{d}{h}\right)^{-d/2} \! >\; 
   \exp\left(-\frac{d^2}{2h}\right)\raisedot
\end{equation}
Now $h \ge h_0(d)$ implies both $d^2 < h$ and
$dh^{1/2} \ge d!e/c^d + d^2$; since $c < 0.9$ the latter
inequality implies 
\begin{equation}	\label{eq:gamma}
\frac{0.9d}{ch^{1/2}} > \frac{d!e}{c^dh} + \frac{d^2}{h}\,\raisedot
\end{equation}
{From} $d^2 \le h$ and the inequalities~\eqref{eq:alpha}--\eqref{eq:gamma},
we have
\begin{eqnarray*}
E(\det(F+I)) &>& c^dn^{d/2}
	\left(1 + \frac{0.9d}{ch^{1/2}} - \frac{d!e}{c^d h}\right)
		\exp\left(-\frac{d^2}{2h}\right)\\
&>& c^dn^{d/2}\left(1 + \frac{d^2}{h}\right)
	\exp\left(-\frac{d^2}{2h}\right) > c^dn^{d/2}.
\end{eqnarray*}
This proves the existence of matrices
$B$ and $C$ such that
$\det(F+I) > c^dn^{d/2}$.

To complete the proof,
we choose the off-diagonal elements of $D$, in an arbitrary order, in such a
manner that $\det(F-D) \ge \det(F+I)$. 
This is always possible, since $\det(F-D)$ is
a linear function of each off-diagonal element $d_{ij}$ considered
separately, so at least one of the choices $d_{ij} = +1$ and $d_{ij} = -1$ 
does not reduce the determinant. 
The inequality~\eqref{eq:first_uncond_bd} 
now follows from Lemma~\ref{lemma:Schur}.
\end{proof}
\begin{remark}	\label{remark:complexandconf}
{\rm 
A variant of Theorem~\ref{thm:unconditional_bound}
arises if we start, not from
an $h \times h$ Hadamard matrix, but from an 
$h \times h$ {\it conference} matrix\footnote{Similarly
for weighing matrices~\cite{Craigen-weighing}, which are also scalar
multiples of orthogonal matrices.}, 
that is a $\{0,\pm1\}$-matrix $C$, with ${\rm diag}(C) = 0$,
satisfying $CC^T = (h-1)I$. To prove the variant, we need only minor
alterations to Lemma~\ref{lemma:u_sum} and to the proof of 
Theorem~\ref{thm:unconditional_bound}.
Using this variant, we can improve the constant\footnote{This constant
occurs in the statement that the Ehlich upper bound~\cite{Ehlich64b}
for $D(n)$ in the case $n \equiv 3 \bmod 4$ is attained
\emph{up to a constant factor} infinitely often.}
in Theorem~C of Neubauer and
Radcliffe~\cite{NR} from $0.3409$ to $0.4484$.
Another interesting variant allows all
matrices to have entries from the set $\{\pm 1, \pm i\}$;
then a $4$-sided ``coin'' and $4$-valued
``sign'' function need to be used.
}
\end{remark}
\begin{corollary}	\label{cor:small_d}
If $1\le d\le 3$, $h \in {\Had}$, $h \ge 4$, and $n = h+d$, then
\begin{equation}	\label{eq:D_nh_ineq}
\frac{D(n)}{h^{h/2}} > \left(\frac{2n}{\pi}\right)^{d/2}
\end{equation}
and
\begin{equation}	\label{eq:D_nh_ineq2}
\Dbar(n) > \left(\frac{2}{\pi e}\right)^{d/2}.
\end{equation}
\end{corollary}
\begin{proof}
First consider the inequality~\eqref{eq:D_nh_ineq}.
This follows from Theorem~\ref{thm:unconditional_bound} if 
$h \ge h_0(d)$.  
The inequality~\eqref{eq:better_ineq} in the
proof of Theorem~\ref{thm:unconditional_bound} covers
all cases with $h \ge 16$ and $d \le 3$, so
we only need check the cases $h \in \{4,8,12\}$
and use the known values (see for example~\cite{Orrick-www}) 
of $D(5),\ldots,D(15)$.

The inequality~\eqref{eq:D_nh_ineq2}
follows from~\eqref{eq:D_nh_ineq} and
Lemma~\ref{lemma:ineq1} (with $\alpha = d$).
\end{proof}
\begin{remark}	\label{remark:assume_H}
{\rm 
If the Hadamard conjecture is true, then for 
$4 < n \not\equiv 0 \pmod 4$, we can take
$h = 4\lfloor n/4\rfloor$ and $d = n-h \le 3$ in
Corollary~$\ref{cor:small_d}$. Thus, 
\[1 > \Dbar(n) > 
 \left(\frac{2}{\pi e}\right)^{d/2} \!\!\ge\;
 \left(\frac{2}{\pi e}\right)^{3/2} \!>\; 0.1133\,.\]
}
\end{remark}
The following corollary does not assume the Hadamard conjecture,
but it does require $h$ to be sufficiently large.
\begin{corollary}	\label{cor:cor1}
Assume that  $d>0$, $h \in {\Had}$, and $h \ge h_0(d)$, where $h_0(d)$ is as in
Theorem~$\ref{thm:unconditional_bound}$.
If $n = h+d$, then
\[\Dbar(n) > \left(\frac{2}{\pi e}\right)^{d/2}.\]
\end{corollary}
\begin{proof}
This follows from Theorem~\ref{thm:unconditional_bound} and
Lemma~\ref{lemma:ineq1} (with $\alpha = d$).
\end{proof}
\begin{corollary}	\label{cor:liminf}
Let $d \ge 0$ be fixed. Then
\[\liminf_{\substack{\\[1pt]n\to\infty\\[1pt]n-d\,\in\,{\Had}}} 
 \Dbar(n) \ge \left(\frac{2}{\pi e}\right)^{d/2}.\]
\end{corollary}
\begin{proof}
The result is trivial if $d = 0$, so suppose that $d \ge 1$. 
Corollary~\ref{cor:cor1} shows that
$\Dbar(n) > (2/(\pi e))^{d/2}$ for $n=h+d$ and all sufficiently
large $h$, so the result follows.
\end{proof}

\begin{corollary}	\label{cor:cor2}
There exist positive constants $\kappa_d$ such that, if $d \ge 0$, 
$h \ge 4$, $h \in {\Had}$, and $n = h+d$, then
$\Dbar(n) \ge \kappa_d$.
\end{corollary}
\begin{proof}
The result is trivial if $d=0$.
Otherwise, define
\[\kappa_d := \min\; 
     \{(2/(\pi e))^{d/2}\} \cup
     \{\Dbar(n)\,|\, n \in \N,\, n-d\in{\Had},\, 4 \le n-d < h_0(d)\}\,.
\]
Since $\kappa_d$ is the minimum of a finite set of positive values,
it is positive, and by Corollary~\ref{cor:cor1} 
it is a lower bound on $\Dbar(n)$.
\end{proof}
\begin{remark}	\label{remark:constants}
{\rm The best (i.e.~largest) possible
values of the constants $\kappa_d$ are unknown, except for the
trivial $\kappa_0 = 1$.
{From} Corollary~\ref{cor:small_d}, we know that 
\begin{equation}	\label{eq:kappa_lower_bd}
\kappa_d \ge \left(\frac{2}{\pi e}\right)^{d/2}
\end{equation}
holds for $d \le 3$, and it is plausible to conjecture that
\eqref{eq:kappa_lower_bd} holds for all $d \ge 0$.
It is unlikely that this inequality is tight, and plausible
that the constant $2/(\pi e)$ could be replaced by some greater value.

If the Hadamard conjecture is true, then we can assume that $d\le 3$
and $\kappa_d \ge ({2}/({\pi e}))^{3/2} > 1/9$.
Hence, it is of interest to mention known upper bounds on the $\kappa_d$
for $d \le 3$.
\begin{enumerate}
\item We have
$\kappa_1 \le \Dbar(9) = 7\times 2^{11}/3^9 < 0.7284$, which is sharper
than the value $(2/e)^{1/2} \approx 0.8578$ given by the Barba
bound~\cite{Barba33} as $n \to \infty$.
\item The Ehlich-Wojtas bound~\cite{Ehlich64a,Wojtas64} in the limit
as $n \to \infty$ shows that
$\kappa_2 \le 2/e < 0.7358$.
\item We have
$\kappa_3 \le \Dbar(11) = 5\times 2^{16}/11^{11/2} < 0.6135$,
which is sharper than the value 
$2e^{-3/2}11^3 7^{-7/2} \approx 0.6545$
given by Ehlich's upper bound~\cite{Ehlich64b} as $n \to \infty$.
\end{enumerate}
}
\end{remark}
We now state and prove three similar theorems.
In the proofs of Theorems~\ref{thm:unconditional_bound2} 
and~\ref{thm:unconditional_bound3}
we need the Schur matrix $F$ to have off-diagonal entries
small compared to its diagonal entries so that we can apply
the determinant bound for diagonally-dominant matrices in
Lemma~\ref{lemma:nice_pert_bound}.
To quantify this we introduce two sets $S_0$ and $S_1$.
Roughly speaking,
$S_0$ is the set of coin-tosses yielding large-enough
diagonal entries of~$F$, and
$S_1$ is the set of coin-tosses yielding too-large
off-diagonal entries of~$F$.
It is necessary to show that $S_0 \backslash S_1 \ne \emptyset$.
We accomplish this by using Lemma~\ref{lemma:ES15.2}
and an independence argument to
show that, with our choice of parameters, 
$S_0$ is not too small and (by using a Hoeffding tail bound)
$S_1$ is smaller than $S_0$. The two theorems differ in the choice
of parameters and largeness/smallness criteria.
Theorem~\ref{thm:unconditional_bound2} gives the sharper bound
but has more restrictive conditions, in particular the
condition $h \ge 16d^3\ln h$.  Theorem~\ref{thm:unconditional_bound3}
relaxes this condition to $h \ge 6d^3$, but at the cost of a weaker
bound on $\Dbar(n)$. Finally, Theorem~\ref{thm:no_exceptions} removes
any restriction on $h$, at the cost of a yet weaker bound (but
still depending only on $d$).

\pagebreak[3]
\begin{theorem}	\label{thm:unconditional_bound2}
Let $d \ge 0$ be given, and let $h \in {\Had}$,
$h \ge 656$, be such that 
\begin{equation}	\label{eq:d_h_ineq}
16d^3 \le \frac{h}{\ln h}\,\raisedot
\end{equation}
If $n = h+d$ and 
\begin{equation}	\label{eq:def_eps}
\varepsilon = \left(\frac{4d\ln h}{h}\right)^{1/2},
\end{equation}
then
\begin{equation}	\label{eq:star}
\frac{D(n)}{h^{n/2}} \ge \left(\frac{2}{\pi}\right)^{d/2}
  \exp(-2.31\,d\varepsilon)\,.
\end{equation}
\end{theorem}
Note that when $d \to \infty$ or $h \to \infty$ then 
\eqref{eq:d_h_ineq}--\eqref{eq:def_eps} imply that
$\varepsilon \to 0+$.
Before proving Theorem~\ref{thm:unconditional_bound2}, we state a 
lemma which collects some of the inequalities that are required.
\begin{lemma}	\label{lemma:assorted}
Under the conditions of Theorem~$\ref{thm:unconditional_bound2}$,
if $d\ge 1$ then the
following six inequalities hold:
\begin{equation}	\label{eq:ass1}
d\varepsilon \le 1/2\,,
\end{equation}
\begin{equation}	\label{eq:ass1b}
\varepsilon \ge 8d/h\,,
\end{equation}
\begin{equation}	\label{eq:ass2}
\varepsilon \le \frac{(2/\pi)^{1/2}-0.5}{1.1} \approx 0.2704\,,
\end{equation}
\begin{equation}	\label{eq:ass3}
2d^2\exp(-\varepsilon^2 h/8) \le (2\varepsilon)^d\,.
\end{equation}
\begin{equation}	\label{eq:ass3b}
1 - \frac{1.1\varepsilon}{c} \ge \exp(-\alpha\varepsilon),
\end{equation}
\begin{equation}	\label{eq:ass4}
g(h) - 1 \ge ((2/\pi)^{1/2} - \varepsilon/10)h^{1/2},
\end{equation}
where $\alpha \approx 1.7262$, $c = \sqrt{2/\pi}$, and
$g(h)$ is as in Lemma~$\ref{lemma:g_bound2}$.
\end{lemma}
\begin{proof}
{From}~\eqref{eq:d_h_ineq} and \eqref{eq:def_eps} we have
\[d^2\varepsilon^2 = \frac{4d^3\ln h}{h} \le \frac{1}{4}\,\raisecomma\]
which proves~\eqref{eq:ass1}.
For~\eqref{eq:ass1b} use $\ln h \ge 1$. 
Thus,
from~\eqref{eq:d_h_ineq}, $h \ge 16d^3 \ge 16d$, so
\[\varepsilon^2 = \frac{4d\ln h}{h} \ge \frac{4d}{h} \ge
\frac{64d^2}{h^2}\,\raisecomma\]
and taking a square root gives~\eqref{eq:ass1b}.
Similarly, 
using \eqref{eq:d_h_ineq} and \eqref{eq:def_eps} gives
\[\varepsilon \le \left(\frac{2\ln h}{h}\right)^{1/3},\]
and the condition $h \ge 656$ then gives
$\varepsilon \le ((2 \ln 656)/656)^{1/3} \approx 0.2704$,
which proves~\eqref{eq:ass2}.

Taking logarithms shows that the inequality~\eqref{eq:ass3}
is equivalent to
${\varepsilon^2 h}/{8} \ge \ln(2d^2) - d\ln(2\varepsilon)$,
and substituting the definition~\eqref{eq:def_eps} of $\varepsilon$ and
simplifying shows that this is equivalent to
\begin{equation}	\label{eq:reduced_ineq}
\ln(16d\ln h) \ge {2\ln(2d^2)}/{d}\,.
\end{equation}
The right side of~\eqref{eq:reduced_ineq} is bounded above by 
$4\sqrt{2}/e \approx 2.081$, but the left side exceeds this value
for all $d \ge 1$ and $h \ge 2$. This completes the proof of~\eqref{eq:ass3}.

To show~\eqref{eq:ass3b}, recall that $\varepsilon \le 0.271$.
Using Lemma~\ref{lemma:exp_ineq}
with $\varepsilon_0 = 0.271$, $\kappa = -1.1/c$, we see that
\eqref{eq:ass3b} is valid for
$\alpha \ge -\ln(1 - 0.271\times 1.1/c)/0.271 \approx 1.7262.$

Finally, for~\eqref{eq:ass4}, Lemma~\ref{lemma:g_bound2} gives
$g(h) > ch^{1/2} + 0.9$. Thus, it is sufficient
to show that $ch^{1/2} + 0.9 - 1 \ge (c - \varepsilon/10)h^{1/2}$,
which is equivalent to $\varepsilon h^{1/2} \ge 1$.  This follows easily
from~\eqref{eq:def_eps}.
\end{proof}

\begin{proof}[Proof of Theorem~\ref{thm:unconditional_bound2}]
As usual, we can assume that $d \ge 1$, as the result is trivial
if $d=0$.
We use the same notation as in the proof of
Theorem~\ref{thm:unconditional_bound}. In particular, $c = \sqrt{2/\pi}$,
$F = CA^{-1}B = h^{-1}CA^TB$, and $M = F - D$, where ${\rm diag}(D) = -I$.

Consider $f_{ij}$ for $i$ fixed and $j \ne i$.  To simplify the notation,
assume that $i=1$ and $j \ne 1$. 
Then
\[f_{1j} = \frac{1}{h} \sum_k \sum_\ell c_{1k}a_{\ell k}b_{\ell j}
	 = \sum_\ell u_{1 \ell} b_{\ell j}\;\;\text{ say},\]
where
\begin{equation}	\label{eq:u_defn}
u_{1 \ell} = \frac{1}{h}\sum_k c_{1k}a_{\ell k}
\end{equation}
and
\[c_{1k} = \text{sgn}(B^TA)_{1k} 
	 = \text{sgn}\left(\sum_\ell b_{\ell 1}a_{\ell k}\right)\,.\]
We see that $c_{1k}$ depends on column $1$ of $B$ and is independent
of the other columns of $B$.  Thus, $f_{1j}$ depends on columns $1$ and $j$ 
of $B$ and is independent of the other columns of $B$. 
Also, from Lemma~\ref{lemma:u_sum},
\begin{equation}	\label{eq:sum_u_j2}
\sum_\ell u_{1 \ell}^2 = 1.
\end{equation}
Consider fixing 
the first column of $B$ and allowing the other columns to vary
uniformly at random. Thus, for fixed $j\in[2,d]$, we can regard
$X_\ell := u_{1 \ell} b_{\ell j}$, $1 \le \ell \le h$, as $h$ 
independent random variables having expectation zero and sum $f_{1j}$.
Also, $|X_\ell| \le  |u_{1\ell}|$.
Thus, by~\eqref{eq:sum_u_j2} and Proposition~\ref{prop:Hoeffding},
we have
\begin{equation}	\label{eq:first_H_bound}
\text{Pr}\left(|f_{1j}| \ge t\right) \le 2e^{-t^2/2}\;\;\text{ for }\;\;
 t > 0.
\end{equation}
The inequality~\eqref{eq:first_H_bound} 
is valid for any choice of the first column of $B$, 
hence it is valid if the first column is chosen at random.
Now allow all columns of $B$ to vary uniformly at random.
Since there are $d(d-1)$ off-diagonal elements $f_{ij}$, it follows 
(without assuming independence of the $f_{ij}$) that%
\footnote{
We could sharpen the argument at this point by using the
Lov\'asz Local Lemma~\cite{EL}
to reduce the right-hand-side of~\eqref{eq:max_off_diag}
to $O(de^{-t^2/2})$, but this would not significantly improve
the final bound~\eqref{eq:star}.
}
\begin{equation}	\label{eq:max_off_diag}
\text{Pr} \left(\max_{i\ne j} |f_{ij}| \; \ge \; t\right)
	\le 2d(d-1)e^{-t^2/2}.
\end{equation}

\noindent
[\emph{Remark}:
The inequality~\eqref{eq:max_off_diag} shows that the off-diagonal elements
of $F$ are usually
``small'', more precisely of order $\sqrt{\log d}$. 
We now consider the diagonal elements and show
that there is a set (not too small) on which they are 
at least $h^{1/2}/2$.]

As in  the proof of Theorem~\ref{thm:unconditional_bound}
(following Best~\cite[Theorem~3]{Best}),
\[E(f_{ii}) = g(h)-1,\] 
where $g(h) \sim ch^{1/2}$ is as in Lemma~\ref{lemma:g_bound2}.
Choose $c_1 < c$ and suppose that $h$ is sufficiently large that
$E(f_{ii}) = g(h)-1 \ge c_1 h^{1/2}$.

Choose $c_2 < c_1$, and consider
$\rho_i := {\rm Pr}\left(f_{ii} \ge c_2 h^{1/2}\right)$. 
By our choice of $C$ and Best~\cite[Thm.~1]{Best}, we have
$0 \le f_{ii} \le h^{1/2}$.  
Thus, by Lemma~\ref{lemma:ES15.2} applied to the random variable
$f_{ii}/h^{1/2}$, we have
\[\rho_i \ge \frac{c_1 - c_2}{1 - c_2}\,\raisedot\]

Note that $f_{ii}$ depends only on the $i$-th column of $B$, so the $f_{ii}$
are independent for $1 \le i \le d$. Thus, if 
$S_0 = \{B\,|\, \min\{f_{ii}| 1 \le i \le d\} \ge c_2h^{1/2}\}$, we have
\[{\rm Pr}(S_0) = \prod_i \rho_i
  \ge \left(\frac{c_1 - c_2}{1 - c_2}\right)^d\,.\]
To be definite take $c_1 = c - \varepsilon/10$ and $c_2 = c_1 - \varepsilon$, 
where $\varepsilon$ is as in the statement of the theorem and, 
{from} Lemma~\ref{lemma:assorted},
$\varepsilon \le (c - 0.5)/1.1 \approx 0.2704$.
Then we have $c_2 = c - 1.1\varepsilon \ge 1/2$,
$\rho_i \ge 2\varepsilon$, and
${\rm Pr}(S_0) \ge (2\varepsilon)^d$.

Let $S_1$ be the set of $B$ for which 
$\max_{i\ne j}|f_{ij}| \;\ge\; t\,.$
{From}~\eqref{eq:max_off_diag}, we have
${\rm Pr}(S_1) \le 2d(d-1)e^{-t^2/2}$.
For the matrix $F$ 
to be ${\rm DD}(\varepsilon)$ on a nonempty set $S_0\backslash S_1$
of choices of $B$, it suffices that
\begin{equation}	\label{eq:conditions}
t \le c_2\varepsilon h^{1/2} \;\;\text{ and }\;\;
  2d(d-1)e^{-t^2/2} < (2\varepsilon)^d\,.
\end{equation}
Thus, choosing $t = c_2\varepsilon h^{1/2}$, it is sufficient that
\begin{equation}	\label{eq:condition1}
2d^2 \exp(-c_2^2\varepsilon^2 h/2) \le (2\varepsilon)^d.
\end{equation}
Since $c_2 \ge 1/2$, part~\eqref{eq:ass3} of 
Lemma~\ref{lemma:assorted} shows
that the inequality~\eqref{eq:condition1} is satisfied.
Thus, Lemma~\ref{lemma:nice_pert_bound} 
applied to $F$ gives
\begin{equation}	\label{eq:detF_bd}
\det(F) \ge (c_2h^{1/2})^d (1-(d-1)^2\varepsilon^2)
\end{equation}
on a nonempty set $S_0\backslash S_1$. 
Since $d\varepsilon \le 1/2$, 
\[1-(d-1)^2\varepsilon^2 \ge 1 - d^2\varepsilon^2 \ge 
  1 - d\varepsilon/2 \ge \exp(-\beta d\varepsilon),\]
where Lemma~\ref{lemma:exp_ineq} gives $\beta = 2\ln(4/3) \approx 0.5755$.
As in the proof of Theorem~\ref{thm:unconditional_bound}, 
we choose the elements of $D$
so that $\det(M) = \det(F-D) \ge \det(F)$.
It follows from Lemma~\ref{lemma:Schur} that
\begin{equation}	\label{eq:ineq_1}
D(n) \ge h^{n/2}c_2^d \exp(-\beta d\varepsilon).
\end{equation}
To complete the proof, use~\eqref{eq:ass3b} of Lemma~\ref{lemma:assorted}.
We have 
$c_2/c \ge \exp(-\alpha\varepsilon)$,
where $\alpha \approx 1.7262$, and 
\begin{equation}	\label{eq:use_in_corollary}
c_2^d \ge c^d\exp(-\alpha d\varepsilon).
\end{equation}
Now the theorem follows from~\eqref{eq:ineq_1},
using $\alpha + \beta < 2.31$.
\end{proof}

\pagebreak[3]
The inequality in the following Corollary~\ref{cor:uncond2} is slightly weaker 
than the inequality in Corollary~\ref{cor:cor1}, but
Corollary~\ref{cor:uncond2} is applicable for smaller
values of~$h$.
Note that $d\varepsilon \le \frac12$ from~\eqref{eq:ass1},
so $\exp(-2.38d\varepsilon) \ge \exp(-1.19) > 0.3$.
\begin{corollary}	\label{cor:uncond2}
Under the conditions of Theorem~$\ref{thm:unconditional_bound2}$,
\[\Dbar(n) > \left(\frac{2}{\pi e}\right)^{d/2}
  \exp(-2.38\,d\varepsilon).\] 
\end{corollary}
\begin{proof}
We can assume that $d > 0$.
From Lemma~\ref{lemma:uncond2} with $\alpha = d$, 
\[\left(\frac{h}{n}\right)^n \,>\, \exp(-d)\exp(-d^2/h).\]
{From}~\eqref{eq:ass1b} of Lemma~\ref{lemma:assorted},
$d^2/h \le d\varepsilon/8$. Thus
\begin{equation}	\label{eq:fudge}
(h/n)^{n/2} \,>\, \exp(-d/2)\exp(-d\varepsilon/16).
\end{equation}
The result now follows from~\eqref{eq:star} and~\eqref{eq:fudge},
since $2.31 + 1/16 < 2.38$.
\end{proof}

Theorem~\ref{thm:unconditional_bound3} weakens the
condition~\eqref{eq:d_h_ineq} on $h$ in
Theorem~\ref{thm:unconditional_bound2} by eliminating the log term; the new
condition is $h \ge 6d^3$.  The cost is a weakening of the result~--
essentially the constant $2/\pi$ in inequality~\eqref{eq:star} is replaced
by a smaller constant, and we have to introduce a factor 
$1 - O(d^3/h)$.
\begin{theorem}	\label{thm:unconditional_bound3}
Let $\delta = 6d^3/h$,
and assume that $\delta \le 1$, $d > 0$, $h \in {\Had}$,
and $n = h+d$.
Then
\[
\frac{D(n)}{h^{n/2}}\, \ge \, (0.594)^d\,(1-0.93\,\delta)
\]
and
\[
\Dbar(n)
 \,\ge\, (0.352)^d\,(1-0.93\,\delta) \,\ge\, 0.07\,(0.352)^d\,.
\]
\end{theorem}

\begin{proof}
We follow the notation and 
proof of Theorem~\ref{thm:unconditional_bound2}, but with a 
different choice of $c_1$, $c_2$ and $\varepsilon$.  

If $d \le 3$ the results follow
from Corollary~\ref{cor:small_d},
so assume that $d \ge 4$.
Since $h \ge 6d^3$, we can assume that $h \ge 384$.

Choose $c_1 = c(1 - 1/(4h))$. By Lemma~\ref{lemma:g_bound2}, 
$g(h) - 1 > c_1 h^{1/2}$. 
Since $h \ge 384$, we have $c_1 \ge 0.797$.
Now choose 
$c_2 = 2c_1 - 1 \ge 0.594$ so that
$\rho_i \ge (c_1-c_2)/(1-c_2) \ge 1/2$
and ${\rm Pr}(S_0) \ge 2^{-d}$.

For the matrix $F$ to be ${\rm DD}(\varepsilon)$ on a nonempty set
$S_0\backslash S_1$ of choices of $B$, it suffices that
\[
t \le c_2\varepsilon h^{1/2} \;\;\text{ and }\;\;
  2d(d-1)e^{-t^2/2} < 2^{-d}\,.
\]
Thus, choosing $t = c_2\varepsilon h^{1/2}$, it is sufficient that
\[
2d(d-1) \exp(-c_2^2\varepsilon^2 h/2) < 2^{-d}\,,
\]
which is equivalent to
\begin{equation}	\label{eq:eps_cond}
\varepsilon^2 > \frac{2d\ln2}{c_2^2h}
	\left(1+\frac{\log_2(2d(d-1))}{d}\right)\,.
\end{equation}
Now $2\ln2/c_2^2 < 3.92$, and~\eqref{eq:eps_cond} is
satisfied if we choose $\varepsilon$ so that
\[
\varepsilon^2 = \frac{3.92d}{h}
  \left(1+\frac{\log_2(2d(d-1))}{d}\right)\,.
\]
To obtain a nontrivial bound from Lemma~\ref{lemma:nice_pert_bound} we need
$(d-1)^2\varepsilon^2 < 1$, or equivalently
\[
\frac{3.92d(d-1)^2}{h}
  \left(1+\frac{\log_2(2d(d-1))}{d}\right)\, <\, 1.
\]
We find numerically\footnote{The 
maximum $5.564\ldots$ occurs at $d=9$.}
that
\[
\max_{d\in\N,\, d\ge 4}\left[\frac{3.92(d-1)^2}{d^2}
  \left(1+\frac{\log_2(2d(d-1))}{d}\right)\right]\, < \, 5.57.
\]
Thus, the condition $h \ge 6d^3$ is sufficient for $F$ to be ${\rm
DD}(\varepsilon)$ on a nonempty set.
Also, we have $(d-1)^2\varepsilon^2 < 5.57\,\delta/6 < 0.93\,\delta$,
so $1-(d-1)^2\varepsilon^2 > 1 - 0.93\,\delta$.
Now~\eqref{eq:detF_bd} and the remainder of the proof
follow as in the proof of
Theorem~\ref{thm:unconditional_bound2},
using
Lemma~\ref{lemma:uncond2} with $\alpha=d$ 
for the inequality involving $\Dbar$,
and observing that 
\[\frac{d^2}{2h} = \frac{d^3}{2h}\cdot\frac{1}{d}
 \le \frac{1}{12d} \le \frac{1}{48}\]
and
\[0.594\exp\left(-\frac12 - \frac1{48}\right) > 0.352.\]
\end{proof}

We now investigate when the conditions of
Theorem~\ref{thm:unconditional_bound3} are satisfied.  First we state
a result of Livinskyi~\cite[Theorem 5.4]{Livinskyi}.  
This result is better for our purposes than the
(asymptotically sharper) result of Livinskyi quoted in
\S\ref{sec:gaps}, as it has
a smaller additive constant.

\begin{proposition}[Livinskyi, Theorem 5.4]	\label{prop:Livinskyi54}
If $p$ is an odd positive integer and
$t = 6\lfloor\frac{1}{26}\log_2\left(\frac{p-1}{2}\right)\rfloor + 11$,
then there exists a Hadamard matrix of order $2^t p$.
\end{proposition}
\begin{corollary}	\label{cor:Livinskyi54}
If $k \in \N$, $q \in \N$, and $1 \le q \le 2^{26k+1}$, then
there exists a Hadamard matrix of order $2^{6k+5}q$.
\end{corollary}
\begin{proof}
If $p$ is odd and $0 \le ({p-1})/{2} < 2^{26k}$, then
Proposition~\ref{prop:Livinskyi54} shows that 
$2^{6k+5}p \in {\Had}$. Thus, if $q = 2^m p$ where
$p$ is odd, the Sylvester construction applied $m$ times shows that
$2^{6k+5}q \in {\Had}$.
\end{proof}

\begin{lemma}	\label{lemma:Liv_gaps}
If $h_i \in {\Had}$, $h_{i+1}\in {\Had}$ are consecutive Hadamard orders
and $h_i \ge 3\times 2^{70}$, then
$6(h_{i+1}-h_i)^3 \le h_i$. 
\end{lemma}
\begin{proof}
{From} Corollary~\ref{cor:Livinskyi54} with $k\ge 3$, the gaps between
consecutive Hadamard orders 
$h_i, h_{i+1} \le 2^{32k+6}$ are at most $2^{6k+5}$, 
and $6\times(2^{6k+5})^3 = 3\times 2^{18k+16}$,
so the result holds
for $h_i, h_{i+1} \in I_k := [3\times 2^{18k+16}, 2^{32k+6}]$.
Now the intervals $I_3 = [3\times 2^{70}, 2^{102}]$,
$I_4 = [3 \times 2^{78}, 2^{134}], \ldots,$ overlap and cover the whole
region $[3\times 2^{70}, \infty)$.
Also, $I_k \cap I_{k+1} = [3\times 2^{18k+34}, 2^{32k+6}]$ is sufficiently
large that the special case $h_i \in I_k$, $h_{i+1} \in I_{k+1}$ causes
no problem, as $h_{i+1} - h_i \le 2^{6(k+1)+5} = 2^{6k+11}$ and both of
$h_i, h_{i+1}$ must belong to one of $I_k$ or $I_{k+1}$. 
\end{proof}

The following lemma shows that the condition
$\delta \le 1$ (that is $6d^3 \le h$) of 
Theorem~\ref{thm:unconditional_bound3} is always satisfied
for $n$ sufficiently large.
\begin{lemma}	\label{lemma:hadregion}
Suppose $n \in \N$, $n \ge 60480$, $h = \max\{x\in {\Had}\,|\, x \le n\}$,
and $d = n-h$.  Then $6d^3 \le h$.
\end{lemma}

\begin{proof}[Sketch of proof]
The proof is mainly based on machine computations, 
so we can only give an outline here.
We split the interval $[60480,\infty)$ into several sub-intervals and
consider each such sub-interval separately. We choose a set of intervals
that overlap slightly in order to avoid any difficulties near the boundaries
between adjacent intervals. (Discussion of such minor details is omitted
below.)

First consider $[60480,2^{31}]$.
We wrote a C program to list a subset $L$
of the known Hadamard orders $h \le 2^{31}$ using several (by no means all) 
known constructions
\cite{Agaian,BH,BDKR,CHK,CSZ,Djokovic10a,Djokovic10b,Hall,HW,HKT,KR,Miyamoto,
      Paley,SebYam,Turyn72,Turyn74,Whiteman,Yang}.
The constructions that we used were:
\pagebreak[3]

\begin{enumerate}
\item
Paley-Sylvester-Turyn: if $p$ is prime (or $p=0$) and $j,k \ge 0$ are
integers, then $h = 2^j(p^k+1) \in \Had$ whenever $4|h$.
\item
Agaian-Sarukhanyan: if $\{4a,4b\} \subset \Had$, then $8ab\in\Had$.
\item
Craigen-Seberry-Zhang: if $\{4a,4b,4c,4d\}\subset\Had$, then
$16abcd\in\Had$.
\item
Twin-Prime construction: if $q$ and $q+2$ are both odd prime powers, 
then $h = (q+2)q+1\in\Had$.
\item
Craigen-Holzmann-Kharaghani~\cite[Cor.~16, pg.~87]{CHK}:
If $q = x+y$ is a sum of two complex Golay numbers $x$ and $y$,
then $h = 8q \in\Had$.
It is known that every integer $g$ of the form
$g =  2^{a-1}\, 6^b\,  10^c\, 22^d\, 26^e$, 
with $a,b,c,d,e \ge 0$ integers, is complex Golay.
For example:
$659 \times 8$,
\hbox{$739 \times 16$},
$971 \times 8$, and
$1223 \times 16$
are all in $\Had$ since
$659=11+648$,
$739\times 2=26+1452$,
$971=968+3$, and
$1223 \times 2=26+2420$.
\item
Miyamoto-I: if $q-1\in\Had$ and $q$ is a prime power, 
then $4q\in\Had$.
\item
Miyamoto-II: if $q$ and $2q-3$ are
prime powers and $q \equiv 3 \bmod 4$, then $8q\in\Had$.
\item
Yamada/Kiyasu: if $q$ is a prime power, $q \equiv 5 \bmod 8$, and
$(q+3)/2\in\Had$, then $4(q+2)\in\Had$.
\item
Small orders: $1$, $2$, and all $h$ divisible by $4$ with
$4 \le h \le 2056$ are in $\Had$ {\it except} perhaps 
$668$, $716$, $892$,  
$1004$, $1132$, $1244$, $1388$, $1436$, $1676$, $1772$, $1916$, 
$1948$, $1964$.
\item
Baumert-Hall-Williamson:
If $w$ is the order of a quadruple of Williamson matrices, and
$4b$ is the order of a Baumert-Hall array, then $4bw\in\Had$.
Known Williamson numbers include all $w$ with $1 \le w \le 64$
{\it except} $\{ 35, 47, 53, 59 \}$.
Known Baumert-Hall numbers $b$ include all $b$ with $1 \le b \le 108$
{\it except} $\{ 97, 103 \}$, and all $b = 2^k + 1$ for 
$k\ge 0$.
\item
Seberry-Yamada~\cite[Cor.~29]{SebYam}:
If $q$ and $2q+3$ both are prime powers, then $w=2q+3$ is a Williamson
order.
For example, $109$ is a Williamson order.
\end{enumerate}

Using the computed $L$ it is easy to check if
any given $n \le 2^{31}$ corresponds to a pair 
$(h,d)$ (with $h, d$ defined as in the
statement of the lemma) such that $6d^3 > h$.
We found that the largest such $n$ is $60480$, corresponding to the open
interval $(60456,60480)$ which does not intersect our list $L$ of known
Hadamard orders (and $6\times 23^3 = 73002 > 60456$).
Thus, we have proved the result claimed for $n \le 2^{31}$.

Now consider the interval $(2\times 10^9, 8\times 10^{18}]$.
There is some overlap with the previous case, since $2\times 10^9 < 2^{31}$.
We use the tables of maximal prime gaps 
at~\cite{OEIS005250,Silva}
(found by e~Silva and others)
for primes $p \le 4 \times 10^{18}$.
The largest of these prime gaps is $1476$.
Using the tables and the fact that $2(p+1)\in {\Had}$ for
every odd prime $p$, 
we find that the claim
holds for $2\times 10^9 < n \le 8\times 10^{18}$.

The tables of maximal prime gaps do not yet extend as far
as $2\times 10^{21}$. Hence we deal with the interval
$(8\times 10^{18}, 4\times 10^{21}]$ in a different manner, but still
using the tables of known maximal prime gaps.

First consider the interval $(7\times 10^{18}, 1.2\times 10^{20}]$.
Since $32(p+1) \in {\Had}$ for prime $p$, it is sufficient to
know prime gaps for primes $p \le 1.2\times 10^{20}/32 < 4\times 10^{18}$.
The largest such prime gap is $1476$, corresponding to a gap between
Hadamard orders of at most $32\times1476 = 47232$.
Since $6\times 47232^3 < 7\times 10^{18}$, the claim holds
for $7\times 10^{18} < n \le 1.2\times 10^{20}$.

Now consider the interval $(10^{20}, 4\times 10^{21}]$.
Since $1000(p+1) \in {\Had}$ for prime $p$, the known prime gaps for
$p \le 4\times 10^{18}$ suffice.  The largest such gap, $1476$, now
corresponds to a gap between Hadamard orders of at most
$1476000$.  Since $6\times 1476000^3 < 10^{20}$, the claim holds
for $10^{20} < n \le 4\times 10^{21}$. 

Finally, since $3\times 2^{70} < 4\times 10^{21}$,
Lemma~\ref{lemma:Liv_gaps} shows that the claim holds
for all $n > 4\times 10^{21}$, which completes the proof.
\end{proof}

\begin{table}[ht]	
\begin{center}		
\begin{tabular}{|c|c|c|c|l|}
\hline
$h$ & $h'$ & $d$ & $p$ & method\\
\hline
$664$ & $672$ & $[5,6]$ & $331$& Paley1\\
$712$ &$720$& $[5,6]$ & $709$& conference\\
$888$ &$896$& $6$	  & $443$& Paley1\\ 
$1000$&$1008$& $6$	  & $499$& Paley1\\
$1128$&$1136$& $6$     & $563$& Paley1\\
$1240$&$1248$& $6$	  & $619$& Paley1\\
$2868$&$2880$& $[8,10]$& $1433$ & Paley2\\
$5744$&$5760$& $[10,14]$&$5749$ & conference\\ 
$10048$&$10064$& $[12,14]$&$5023$& Paley1\\
$23980$&$24000$& $[16,18]$&$23993$ & conference\\
$47964$&$47988$& $[20,22]$&$47963$ & Paley1\\ 		
$53732$&$53760$& $[21,26]$&$53731$ & Paley1\\ 		
$60456$&$60480$& $22$     &$60457$ & conference\\      	
\hline
\end{tabular}
\caption{Exceptional cases in the proof of Theorem~\ref{thm:no_exceptions}.}
\end{center}
\end{table}

By considering a small set of exceptional cases, we now show that
the condition $6d^3 \le h$ of Theorem~\ref{thm:unconditional_bound3} 
can be dropped entirely, if we are satisfied with a slightly
weaker lower bound on $\Dbar(n)$.
\begin{theorem} \label{thm:no_exceptions}
Suppose that $n \in \N$, $h = \max\{x \in {\Had}\,|\,x\le n\}$,
and $d = n-h$. Then \[\Dbar(n) > 0.07\,(0.352)^d > 3^{-(d+3)}.\] 
\end{theorem}
\begin{proof}	
For $0 \le d \le 3$, the result follows from Corollary~\ref{cor:small_d}.
This covers all $n < 668$.
On the other hand,
if $n \ge 60480$, the result follows from 
Theorem~\ref{thm:unconditional_bound3} and Lemma~\ref{lemma:hadregion}.
Thus, we can assume that $668 \le n < 60480$ and $d \ge 4$.

If $n+1\in\Had$ then Theorem~9 of~\cite{rpb249} gives
$\Dbar(n) \ge (4/(ne))^{1/2}$, and for $n < 60480$, $d \ge 4$,
it is easy to verify
that $(4/(ne))^{1/2} > 0.002 > 0.07\,(0.352)^d$.
Thus, if $h$, $h'$ are consecutive (known) Hadamard orders, we only have to
consider the cases $n=h+d$ for $4\le d\le h'-h-2$.

From the output of the C program described in the proof of
Lemma~\ref{lemma:hadregion}, we find that the cases that are
not covered by Lemma~\ref{lemma:hadregion} or the remarks
already made are those listed in {\Tabexceptions},
which gives $32$ cases in $13$ intervals. For each of these $13$
intervals $[h,h']$ we know that $h$ and $h'$ are Hadamard orders, but
we do not \emph{know}\footnote{This may just reflect our ignorance.
Certainly such orders exist if the Hadamard conjecture is true.
In some cases 
we know that they exist via constructions that were not
implemented in our C program.} 
any Hadamard orders in the open interval $(h,h')$, and we need
to verify that the inequality 
\begin{equation}	\label{eq:uniform_bound}
\Dbar(n) > 0.07\,(0.352)^d
\end{equation}
is satisfied for each $n=h+d$ and the values of $d$ listed in the
third column of the table.

Using Magma~\cite{Magma}, we wrote a program 
that implements a randomised algorithm to obtain a lower bound
on $\Dbar(n)$. The program constructs a Hadamard matrix
$A$ of order $h=(p+1)$ or $h=2(p+1)$, where $p$ is an odd prime and in the
first case $p \equiv 3 \bmod 4$, using the appropriate Paley
construction~\cite{Paley}, followed if necessary by the Sylvester
construction~\cite{Sylvester}. The program then generates a border of width
$d$ to obtain a matrix
$\widetilde{A}$ of order~$n$, as in the proof of
Theorem~\ref{thm:unconditional_bound}, and computes 
$|\det(\widetilde{A})|/n^{n/2}$ by computing the determinant
of the Schur complement of $A$ in $\widetilde{A}$
and using Lemma~\ref{lemma:Schur}. If desired, several independent random trials
can be performed to improve the lower bound.

Using our Magma program with the primes $p$ listed in the fourth column
of {\Tabexceptions}, we were able to show that
the inequality~\eqref{eq:uniform_bound} holds
for all the cases labelled ``Paley1'' or ``Paley2''.  In fact, a few trials
of our randomised algorithm were sufficient to show that the stronger
inequality 
\begin{equation}	\label{eq:conj_ineq}
\Dbar(n) > \left(\frac{2}{\pi e}\right)^{d/2}
\end{equation}
holds in these cases (this is not surprising,
in view of Corollaries~\ref{cor:small_d} and~\ref{cor:cor1}).

For the intervals $[h,h']$ labelled ``conference'' in {\Tabexceptions}, 
there is
no prime $p$ for which $h=p+1$ or $2(p+1)$, but there is a prime $p$
(given in the fourth column of the table) which can be used to 
construct a conference matrix of order $p+1$ close to $h$.
Using a slight modification of our Magma program, we can use this
conference matrix to obtain lower bounds on $\Dbar(n)$ for $n \ge p+1$
(see Remark~\ref{remark:complexandconf}).  In this way we showed that
the inequality~\eqref{eq:conj_ineq} holds for all the intervals
labelled ``conference'' with the exception of the interval $[712,720]$.
Here there is no suitable prime inside the interval, so we
use $p=709 < h=712$, thus obtaining weaker lower bounds.
However, we still obtain $\Dbar(712+d) > 0.352^d$ for $d \in \{5,6\}$
by this method, and this bound is sufficient since it is stronger 
than the desired inequality~\eqref{eq:uniform_bound}. 

There is one further point to consider. We illustrate it for
the interval $[5744,5760]$ of length $16$. It is
\emph{possible} that $5748$, $5752$ and/or $5756$ are Hadamard orders
(although we do not at present know how to construct Hadamard matrices
of these orders). Thus, we need to check that our lower bound on 
$\Dbar(n)$ holds for $h=5748$, $d=10$, $n=h+d=5758$
(and other similar cases).
The prime $p=5749$ gives a conference matrix of order $5750$.
Using this conference matrix, our program shows that 
$\Dbar(5758) > 0.002115 > (2/(\pi e))^{10/2}$, so~\eqref{eq:conj_ineq}
is satisfied.  The other, similar, cases that arise if a Hadamard
order exists in the interior of any of the intervals listed in
{\Tabexceptions} can be covered by one of the arguments that we have
already used.  Thus, the inequality~\eqref{eq:uniform_bound} always holds
for the exceptional cases listed in {\Tabexceptions}.
\end{proof}

\pagebreak[3]


\begin{thebibliography}{99}
 
{\small
\bibitem{Agaian} 
S. S. Agaian,    
\emph{Hadamard Matrices and their Applications},
Lecture Notes in Mathematics, Vol.\ 1168, Springer-Verlag, 1985.

\bibitem{AS}
N. Alon and J. H. Spencer, 
\emph{The Probabilistic Method}, 3rd edn., Wiley, 2008.

\bibitem{Putnam74}
Anonymous, 
Putnam 
Competition, 1974,
\url{http://www.math-olympiad.com/35th-putnam-mathematical-competition-1974-problems.htm}.

\bibitem{BHP}
R. C. Baker, G. Harman and J. Pintz,
The difference between consecutive primes, II,
\emph{Proc.\ London Mathematical Society} \textbf{83} (2001), 532--562.

\bibitem{Barba33}
G. Barba, Intorno al teorema di Hadamard sui determinanti a valore massimo,
\emph{Giorn. Mat. Battaglini} {\bf{71}} (1933), 70--86.

\bibitem{BH}
L. D. Baumert and M. Hall, 
Hadamard matrices of the Williamson type,
\emph{Maths.\ of Comput.\ }\textbf{19} (1965), 442--447.

\bibitem{BDKR}
D. Best, 
D. \v{Z}. {\DJ}okovi\'c (Djokovi\'c), H. Kharaghani and H. Ramp,
{Turyn-type sequences: classification, enumeration and construction},
\emph{J.~Combin.\ Designs} \textbf{21} (2013), 24--35.
Also arXiv:1206.4107v1, 19 June 2012.

\bibitem{Best}
M. R. Best, The excess of a Hadamard matrix,
\emph{Nederl.\ Akad.\ Wetensch.\ Proc.\ Ser. A} \textbf{80}
$=$ 
\emph{Indag.\ Math.\ }\textbf{39} (1977), 
357--361.

\bibitem{Magma}
W. Bosma, J. Cannon and C. Playoust,
The Magma algebra system. I. The user language, 
\emph{J.\ Symbolic Comput.}, \textbf{24} (1997), 235--265.

\bibitem{rpb249}
R. P. Brent and J. H. Osborn,
{General lower bounds on maximal determinants of binary matrices},
\emph{The Electronic Journal of Combinatorics} \textbf{20}(2), 2013, \#P15, 
12~pp.
Also arXiv:1208.1805v6, 14 April 2013.

\bibitem{rpb226}
R. P. Brent and P. Zimmermann, \emph{Modern Computer Arithmetic},
Cambridge University Press, 2010. 

\bibitem{BS}
T. A. Brown and J. H. Spencer,
Minimization of $\pm1$ matrices under line shifts,
\emph{Colloq.\ Math.} 
\textbf{23} (1971), 165--171.
Erratum \emph{ibid} pg.~177.

\bibitem{CL65}
G. F. Clements and B. Lindstr\"om,
A sequence of $(\pm 1)$-determinants with large values,
\emph{Proc.\ Amer.\ Math.\ Soc.\ }\textbf{16} (1965), 548--550.

\bibitem{Cohn63}
J. H. E. Cohn,
On the value of determinants,
\emph{Proc.\ Amer.\ Math.\ Soc.\ }\textbf{14} (1963), 581--588.

\bibitem{CRC1}
C. J. Colbourn and J. H. Dinitz, editors,
\emph{The CRC Handbook of Combinatorial Designs},
CRC Press, Boca Raton, 1996.

\bibitem{CRC2}
C. J. Colbourn and J. H. Dinitz, editors,
\emph{The CRC Handbook of Combinatorial Designs},
second edition, CRC Press, Boca Raton, 2007.

\bibitem{Craigen1}
R. Craigen,
Signed groups, sequences, and the asymptotic existence of Hadamard
matrices,
\emph{J.\ Comb.\ Theory}, Series A \textbf{71} (1995), 241--254.

\bibitem{Craigen2}
R. Craigen, 
Hadamard matrices and designs,
in \cite[Chapter IV.24]{CRC1}.

\bibitem{Craigen-weighing}
R. Craigen, 
Weighing matrices and conference matrices,
in \cite[pp.~496--504]{CRC1}.

\bibitem{CHK}
R. Craigen, W. H. Holzmann, and H. Kharagani,
Complex Golay sequences: structure and applications,
\emph{Discrete Mathematics} \textbf{252}, 
(2002), 73--89.

\bibitem{CK}
R. Craigen and H. Kharagani,
Hadamard matrices and Hadamard designs,
in \cite[Chapter V.1]{CRC2}.

\bibitem{CSZ}
R. Craigen, J. Seberry and X. Zhang,
Product of four Hadamard matrices,
\emph{J.\ Comb.\ Theory}, Series A \textbf{59} 
(1992), 318--320.


\bibitem{Cramer}
H. Cram\'er, 
On the order of magnitude of the difference between consecutive prime
numbers, \emph{Acta Arithmetica} \textbf{2} (1936), 23--46.

\bibitem{Djokovic10a}
D. \v{Z}. {\DJ}okovi\'c (Djokovi\'c), 
Hadamard matrices of small order and Yang conjecture,
\emph{J.\ of Combinatorial Designs} \textbf{18} 
(2010), 254--259.

\bibitem{Djokovic10b}
D. \v{Z}. {\DJ}okovi\'c (Djokovi\'c), 
{Small orders of Hadamard matrices and base sequences},
\emph{International Mathematical Forum} \textbf{6} (2011), 
3061--3067.

\bibitem{Ehlich64a}
H. Ehlich, Determinantenabsch\"atzungen f\"ur bin\"are Matrizen,
\emph{Math. Z.} {\bf{83}} (1964), 123--132.

\bibitem{Ehlich64b}
H. Ehlich, Determinantenabsch\"atzungen f\"ur bin\"are Matrizen mit 
$n \equiv 3 \bmod 4$,
{\em Math. Z.} {\bf{84}} (1964), 438--447.

\bibitem{EL}
P. Erd\H{o}s and L. Lov\'asz,
Problems and results on $3$-chromatic hypergraphs and some related
questions, in \emph{Infinite and Finite Sets}
(A. Hajnal, R. Rado and V. T. S\'os, eds.), North-Holland,
Amsterdam, 1975, 609--628.

\bibitem{ES}
P. Erd\H{o}s and J. Spencer,
\emph{Probabilistic Methods in Combinatorics},
Akad\'emiai Kiad\'o, Budapest, 1974.
Also published by Academic Press, New York, 1974.

\comment{
Warren Smith writes (20121015):

Their lemma 15.1 is essentially the same as my original approach using
Hoeffding-type tail bound.

They also consider MINimizing the excess in thm 5.1,
showing for any NxN sign matrix the excess can always be made <=2 if N
even, <=1 if
N odd, citing

J.Komlos & M.Sulyok 1969
in some unobtainable Hungarian conference: a better cite to it from MR is
On the sum of elements of +-1 matrices. pp.721-728 in
Combinatorial theory and its applications, II (Proc. Colloq.,
Balatonfured, 1969)
North-Holland, Amsterdam, 1970.

The Komlos-Sulyok proof was only for sufficiently large N and
nonconstructive. It later was improved to hold for all N and be
constructive with efficient algorithm, by

Jozsef Beck & Joel Spencer:
Balancing matrices with line shifts,
Combinatorica 3,3-4 (1983) 299-304.

Then in
Balancing matrices with line shifts. II. Irregularities of partitions
(Fert\"od, 1986), 23–37,
Algorithms Combin. Study Res. Texts, 8, Springer, Berlin, 1989.
they went further by allowing the entries of the matrix to be reals x
with -1<=x<=+1.
} 

\bibitem{FK}
N. Farmakis and S. Kounias, 
The excess of Hadamard matrices and optimal designs,
\emph{Discrete Mathematics} \textbf{67} (1987), 165--176.

\bibitem{Gerschgorin}
S. Gerschgorin, \"Uber die Abgrenzung der Eigenwerte einer Matrix,
\emph{Izv.\ Akad.\ Nauk.\ USSR 
Otd.\ Fiz.-Mat.\ Nauk}
\textbf{6} (1931), 749--754.

\bibitem{Hadamard}
J. Hadamard,
R\'esolution d'une question relative aux d\'eterminants,
\emph{Bull. des Sci. Math.} \textbf{17} (1893), 240--246.

\bibitem{Hall}
M.~Hall, Jr., 
\emph{Combinatorial Theory}, 2nd edition, 
Wiley Classics Library, 1986.

\bibitem{HW}
A. Hedayat and W. D. Wallis,
Hadamard matrices and their applications,
\emph{Annals of Statistics} \textbf{6} (1978), 1184--1238.

\bibitem{Hoeffding}
W. Hoeffding, 
Probability inequalities for sums of bounded random variables,
\emph{J.\ Amer.\ Statistical Association} \textbf{58} (1963), 
13--30.

\bibitem{HKT}
W. H. Holzmann, H. Kharaghani, and B. Tayfeh-Rezaie,
Williamson matrices up to order 59,
\emph{Designs Codes Cryptogr.\ }\textbf{46} (2008), 343--352.

\bibitem{Horadam}
K. J. Horadam,
\emph{Hadamard Matrices and their Applications},
Princeton University Press, Princeton, New Jersey, 2007.


\bibitem{KR}
H. Kharaghani and B. Tayfeh-Rezaie,
A Hadamard matrix of order $428$,
\emph{J.\ of Combinatorial Designs} \textbf{13} (2005), 435--440.


\bibitem{KMS00}
C. Koukouvinos, M. Mitrouli and J. Seberry,
Bounds on the maximum determinant for $(1,-1)$ matrices,
\emph{Bulletin of the Institute of Combinatorics and its Applications}
\textbf{29} (2000), 39--48.

\bibitem{LL}
W. de Launey and D. A. Levin, 
$(1,-1)$-matrices with near-extremal properties, 
\emph{SIAM J.\ Discrete Math.\ }\textbf{23} (2009), 1422--1440.

\bibitem{Livinskyi}
I. Livinskyi, 
\emph{Asymptotic existence of Hadamard matrices},
M.Sc.\ thesis, University of Manitoba, 2012.
\url{http://hdl.handle.net/1993/8915}

\bibitem{MS}
V. Maz'ya and T. Shaposhnikova, 
\emph{Jacques Hadamard, A Universal Mathematician},
History of Mathematics, Vol.~14, AMS and LMS, 1998.

\bibitem{Miyamoto}
M. Miyamoto, 
A construction of Hadamard matrices,
\emph{J.~Combin.\ Theory Ser.\ A} \textbf{57} (1991), 86--108.

\bibitem{Muir26}
T. Muir,
Hadamard's approximation theorem since 1900,
\emph{Trans.\ Royal Soc.\ of South Africa} \textbf{13} (1925), 299--308.


\bibitem{NR}
M. G. Neubauer and A. J. Radcliffe,
The maximum determinant of $\pm1$ matrices,
\emph{Linear Algebra Appl.\ }\textbf{257} (1997), 289--306.

\bibitem{OEIS005250}
OEIS Foundation Inc., 
\emph{The On-Line Encyclopedia of Integer Sequences}, 2012,
\url{http://oeis.org/A005250/a005250.txt}

\bibitem{OS07}
W. P. Orrick and B. Solomon, 
Large determinant sign matrices of order $4k+1$,
\emph{Discrete Mathematics} \textbf{307} (2007), 226--236. 

\bibitem{Orrick-www} 
W. P. Orrick and B. Solomon,
\textsl{The Hadamard maximal determinant problem},
\url{http://www.indiana.edu/~maxdet/}

\bibitem{Paley}
R. E. A. C. Paley, 
On orthogonal matrices,
\emph{J.~Mathematics and Physics} 
\textbf{12} (1933), 311--320.

\bibitem{Orrick-prog}
T. Rokicki, I. Kazmenko, J-C. Meyrignac, W. P. Orrick, V. Trofimov and
J. Wroblewski,
\emph{Large determinant binary matrices: 
results from Lars Backstrom's programming contest},
unpublished report, July 31, 2010.

\bibitem{Schur}
I. Schur, 
Neue Begr\"undung der Theorie der Gruppencharaktere,
\emph{Sitzungsberichte der K\"oniglich Preu{\ss}ischen Akademie der 
Wissen\-schaften zu Berlin}, 1905, 406--432.

\bibitem{SebYam}
J. Seberry and M. Yamada, 
{On the products of Hadamard, Williamson and
other orthogonal matrices using M-structures}, 
\emph{J.\ Combin.\ Math.\ Comb.\ Comput.\ }\textbf{7} (1990), 97--137.

\bibitem{Shanks64}
D. Shanks,
On maximal gaps between successive primes,
\emph{Math.\ Comput.\ }\textbf{18} (1964), 646--651.

\bibitem{Silva}
T. O. e~Silva, 
\emph{Gaps between consecutive primes},
\url{http://www.ieeta.pt/~tos/gaps.html}.

\bibitem{Smith-epsilon}
W. D. Smith,
\emph{Asymptotic density of Hadamard matrices}, 
in preparation.


\bibitem{Sylvester}
J. J. Sylvester,
Thoughts on inverse orthogonal matrices, simultaneous sign successions, and
tesselated pavements in two or more colours, with applications to Newton's
rule, ornamental tile-work, and the theory of numbers,
\emph{Philosophical Magazine}
\textbf{34} (1867), 461--475.

\bibitem{Turyn72}
R. J. Turyn, 
An infinite class of Williamson matrices,
\emph{J.~Combin.\ Theory A}, \textbf{12} 
(1972), 319--321.

\bibitem{Turyn74}
R. J. Turyn,	
Hadamard matrices, Baumert-Hall units, four-symbol
sequences, pulse compression, and surface wave encodings,
\emph{J.~Combin.\ Theory A} \textbf{16} 
(1974), 313--333.

\bibitem{Varga}
R. S. Varga, 
\emph{Ger\v{s}gorin and His Circles},
Springer Series in Computational Mathematics, Vol.~36, 2004.

\bibitem{Seberry}
J. Seberry Wallis,
On the existence of Hadamard matrices,
\emph{J.\ Comb. Theory} \textbf{21} (1976), 188--195.

\bibitem{Whiteman}
A. L. Whiteman, 
An infinite family of Hadamard matrices of Williamson type, 
\emph{J.~Combin.\ Theory A} \textbf{14} (1973), 334--340.

\bibitem{Wojtas64}
M. Wojtas,		
On Hadamard's inequality for the determinants of order non-divisible by $4$,
\emph{Colloq. Math.} \textbf{12} (1964), 73--83.

\bibitem{XXS}
T. Xia, M. Xia and J. Seberry,
Regular Hadamard matrices, maximum excess and SBIBD,
\emph{Australasian J.\ Combinatorics} \textbf{27} (2003), 263--275.

\bibitem{Yang}
C. H. Yang,
Hadamard matrices constructible using circulant submatrices,
\emph{Maths.\ of Comput.\ }\textbf{25} 
(1971), 181--186; corrigendum \emph{ibid} \textbf{28} (1974), 1183--1184.
} 

\end{thebibliography}
\end{document}